\documentclass[reqno,12 pt]{amsart}
\usepackage{graphicx}
\usepackage{subfigure}
\usepackage{amsmath}
\usepackage{amsfonts}
\usepackage{amssymb}
\usepackage{amsthm}
\usepackage{cite}

\usepackage{float}

\providecommand{\U}[1]{\protect\rule{.1in}{.1in}}
\providecommand{\U}[1]{\protect\rule{.1in}{.1in}}
\providecommand{\U}[1]{\protect\rule{.1in}{.1in}}
\theoremstyle{plain}
\newtheorem{theorem}{Theorem}
\theoremstyle{definition}
\newtheorem{remark}[theorem]{Remark}

\newtheorem{algorithm}[theorem]{Algorithm}

\numberwithin{equation}{section} \numberwithin{theorem}{section}
\newtheorem{definition}[theorem]{Definition}

\title[New projection-based scheme]{A generalized projection-based scheme for solving convex constrained optimization problems}
\author[A.~Gibali, K.-H.~K\"{u}fer, D.~Reem and P.~S\"{u}ss]{Aviv Gibali, Karl-Heinz K\"{u}fer, Daniel Reem and Philipp S\"{u}ss}
\address[A.~Gibali]{$\boxtimes$ Department of Mathematics, ORT Braude College, P.O. Box 78,
Karmiel 2161002, Israel}
\email{avivg@braude.ac.il}
\address[K.-H.~K\"{u}fer]{Optimization Department, Fraunhofer - ITWM, Kaiserslautern 67663, Germany}
\email{Karl-Heinz.Kuefer@itwm.fraunhofer.de}
\address[D.~Reem]{The Technion - Israel Institute of Technology, 3200003 Haifa, Israel}
\email{dream@technion.ac.il}
\address[P.~S\"{u}ss]{Optimization Department, Fraunhofer - ITWM, Kaiserslautern 67663, Germany}
\email{philipp.suess@itwm.fraunhofer.de}

\date{March 15, 2018. Accepted to publication in \textit{Computational Optimization and Applications}. https://link.springer.com/article/10.1007/s10589-018-9991-4}

\begin{document}

\maketitle

\begin{abstract}
In this paper we present a new algorithmic realization of a projection-based scheme for general convex
constrained optimization problem. The general idea is to transform the
original optimization problem to a sequence of feasibility problems by
iteratively constraining the objective function from above until the
feasibility problem is inconsistent. For each of the feasibility problems one
may apply any of the existing projection methods for solving it. In particular, the
scheme allows the use of subgradient projections and does not require exact
projections onto the constraints sets as in existing similar methods.

We also apply the newly introduced concept of superiorization to optimization formulation and compare its performance to our scheme. We provide some numerical results for convex quadratic test problems
as well as for real-life optimization problems coming from medical treatment
planning.\medskip

\bigskip

\textbf{$2010$ MSC}: 65K10, 65K15, 90C25\smallskip

\end{abstract}

\vspace{-2ex}

\section{Introduction}

In this paper we are concerned with a general convex optimization problem. Let
$f:%
\mathbb{R}
^{n}\rightarrow%
\mathbb{R}
$ and $\left\{  g_{i}:%
\mathbb{R}
^{n}\rightarrow%
\mathbb{R}
\right\}  _{i\in I}$, for $I=\{1,\ldots,m\}$ be convex functions. We wish to solve the following convex optimization
problem.%
\begin{gather}
\min f(x)\nonumber\\
\text{such that }g_{i}(x)\leq0\text{ for all }i\in I. \label{Problem:1}%
\end{gather}

The literature on this problem is vast and there exist many different
techniques for solving it, see e.g., \cite{BSS06,Bertsekas99,
BoydVandenberghe04} and the many references therein. As a special case when
$f\equiv0$, (\ref{Problem:1}) reduces to find a point in the convex set,%
\begin{equation}
\boldsymbol{C}:=\cap_{i\in I}C_{i}=\cap_{i\in I}\left\{  x\in%
\mathbb{R}
^{n}\mid g_{i}(x)\leq0\right\}  \neq\emptyset. \label{eq:Bold_C}%
\end{equation}
$\boldsymbol{C}$ is called the \textit{feasibility set} of (\ref{Problem:1}).
In general, the problem of finding a point in the intersection of convex sets
is known as the \textit{Convex Feasibility Problem} (CFP) or \textit{Set
Theoretic Formulation}. Many real-world problems in various areas of
mathematics and of physical sciences can be modeled in this way; see
\cite{cap88} and the references therein for an early example. More work on the
CFP can be found in \cite{Byrne, byrne04, cdh10}.

In case where all the $\left\{  g_{i}\right\}  _{i\in I}$ are linear and only
equalities are considered in $C_{i}$, meaning that all the $C_{i}\ $ are
hyper-planes, the CFP reduces to a system of linear equations; Kaczmarz
\cite{Kaczmarz37} and Cimmino \cite{Cimmino38} in the 1930's, proposed two
different projection methods, sequential and simultaneous projection methods
for solving a system of linear equalities. These methods were later extended
for solving systems of linear inequalities, see \cite{Agmon54, MS54}. Today,
projection methods are more general and are applied to solve general convex
feasibility problems.

In general, projection methods are iterative procedures that employ
projections onto convex sets in various ways. They typically follow the
principle that it is easier to project onto the individual sets (usually
closed and convex) instead of projecting onto other derived sets (e.g. their
intersection). The methods have different algorithmic structures, of which
some are particularly suitable for parallel computing, and they demonstrate
desirable convergence properties and good initial behavior patterns, for more
details see for example \cite{cccdh11}.

In last decades, due to their computational efficiency, projection methods
have been applied successfully to many real-world applications, for example in
imaging, see Bauschke and Borwein \cite{bb96} and Censor and Zenios \cite{CZ97}, in transportation problems \cite{cz91b,bk13}, sensor networks \cite{bh06}, in radiation therapy treatment planning \cite{cap88,ccmzkh10},
in resolution enhancement \cite{coo03}, in graph matching \cite{ww04}, in matrix balancing \cite{cz91, Amelunxen11}, in radiation therapy treatment planning, in resolution enhancement \cite{cccdh11, BC11, cc14}, to name but a few.
Their success is based on their ability to handle huge-size problems since they do not require storage or inversion of the full constraint matrix. Their algorithmic structure is either sequential or simultaneous, or in-between, as in the block-iterative projection methods or string-averaging projection methods which naturally support parallelization. This is one of the reasons that this class of methods was called \textquotedblleft Swiss army knife\textquotedblright,
see \cite{bk13}.

Following the above we aim to apply different projection methods for solving
(\ref{Problem:1}); in order to do that we first put (\ref{Problem:1}) into
epigraph form.%
\begin{gather}
\min t\nonumber\\
\text{such that }t\in%
\mathbb{R}%
\nonumber\\
\text{and for some }x\in\boldsymbol{C}\text{ one has }f(x)\leq t.
\label{Problem:2}%
\end{gather}
Denote by $t^{\ast}$ the optimal value of (\ref{Problem:2}) which we assumed is attained
and finite. Now, a natural idea for solving
(\ref{Problem:2}) is to construct a decreasing sequence $\left\{
t_{k}\right\}  $ such that $t_{k}\rightarrow t^{\ast}$ and at each step, for a
fixed $t_{k}$, to solve a corresponding CFP; see also \cite[Subsection
2.1.2]{Bertsekas99}. Formally this can be phrased as follows. Set
$t_{-1}=\infty$ and at the $k$-th step, with $k\geq0$, solve the following
problem:%
\begin{equation}
\text{find a point }x^{k}\text{ such that}\left\{
\begin{array}
[c]{l}%
f(x^{k})\leq t_{k-1}\\
\text{and}\\
g_{i}(x^{k})\leq0\text{ for all }i\in I.
\end{array}
\right.  \label{Problem:k_CFP}%
\end{equation}
Once a feasible point $x^{k}$ is obtained, $t_{k}$ is updated according to the
formula%
\begin{equation}
t_{k}=f(x^{k})-\varepsilon_{k}, \label{eq:epsilon_decrease}%
\end{equation}
where $\varepsilon_{k}>0$ is some user chosen constant; in the numerical
experiments in Section \ref{sec:Numerical_Experiments} we use $\epsilon
_{k}=0.1|f(x^{k})|$ whenever $|f(x^{k})|>1$ and $\epsilon_{k}=0.1$
otherwise\textbf{.} For solving these CFPs at each $k$-th step, we apply
different projection methods based on the \textit{Cyclic Subgradient
Projections} \textit{Method} (CSPM) \cite{cl81, cl82} and thus obtain several
algorithmic realizations of this general scheme.
In Subsection \ref{sec:Convergence} we discuss the issue of convergence to an approximate optimal solution of (\ref{Problem:1}) (which we call an $\varepsilon$-optimal solution.)

It might happen that the objective function decreases by $\varepsilon_{k}$ after
each step. Therefore, we not only wish to solve each of the CFPs (\ref{Problem:k_CFP}), but also end
up with a \textit{Slater point}, that is a point that solves (\ref{Problem:k_CFP}) with strict inequalities at each step
to maximize the decrease in $t_{k}$. To this end, we will make use of over
relaxation parameters in the CSPM for one realization of the scheme, but also
apply the newly introduced \textit{Superiorization} idea \cite{cdh10}, where
perturbations are used in the CSPM to steer the algorithm into the interior of
the objective level set after the $k$-th CSP. Our major contribution in this work, Section \ref{sec:Numerical_Experiments}, is the applicability of the scheme and its comparability with some existing results as demonstrated and compared intensively on benchmark quadratic programming problems and medical therapy.

The paper is organized as follows. In Section \ref{sec:Prelim} we
present several projection methods and definitions which will be useful for
our analysis. Next in Section \ref{sec:GeneralOpt} our general scheme for
solving convex optimization problems are presented and analyzed. Later in
Section \ref{sec:Numerical_Experiments} numerical experiments illustrating the
different realizations of our scheme are presented and tested for convex
quadratic programming problems and for Intensity-Modulated Radiation Therapy (IMRT). Finally in Section \ref{sec:conclusion}
conclusion and further research directions are presented.

\section{Preliminaries\label{sec:Prelim}}

In this section we provide several projection methods which are relevant to
our results, mainly orthogonal and subgradient projections. We start by presenting several definitions which will be useful for our analysis.

\begin{definition}
A sequence $\left\{  x^{k}\right\}  _{k=0}^{\infty}$ is said to be
\texttt{finite convergent} if $\lim_{k\rightarrow\infty}x^{k}=x^{\ast}$ and
there exists $N\in%
\mathbb{N}
$ such that for all $k\geq N$, $x^{k}=x^{\ast}$.
\end{definition}

Let $C$ be non-empty, closed and convex set in the Euclidean space $%
\mathbb{R}
^{n}$. Assume that the set $C$ can be represented as%
\begin{equation}
C=\left\{  x\in%
\mathbb{R}
^{n}\mid c(x)\leq0\right\}  , \label{eq:C}%
\end{equation}
where $c:%
\mathbb{R}
^{n}\rightarrow%
\mathbb{R}
$ is an appropriate continuous and convex function. Take, for example,
$c(x)=\operatorname*{dist}(x,C),$ where $\operatorname*{dist}$ is the distance
function; see, e.g., \cite[Chapter B, Subsection 1.3(c)]{hul01}.\medskip

\begin{definition}
For any point $x\in%
\mathbb{R}
^{n}$,\ the \texttt{orthogonal projection} of $x$ onto $C$, denoted by
$P_{C}(x)$ is the closest point to $x$ in $C$, that is,%
\begin{equation}
\left\Vert x-P_{C}\left(  x\right)  \right\Vert \leq\left\Vert x-y\right\Vert
\text{ for all }y\in C.
\end{equation}

\end{definition}

\begin{definition}
Let $c$ be as in the representation of $C$ in (\ref{eq:C}). The set%
\begin{equation}
\partial c(z):=\{\xi\in%
\mathbb{R}
^{n}\mid c(y)\geq c(z)+\langle\xi,y-z\rangle\text{ for all }y\in%
\mathbb{R}
^{n}\}
\end{equation}\label{eq:2.3}
is called the \texttt{subdifferential} of $c$ at $z$ and any element of
$\partial c(z)$ is called a \texttt{subgradient}.
\end{definition}

It is well-known that if $C$ is non-empty, closed and convex, then $P_C(x)$ exists and is unique. Moreover, if $c$ is differentiable at $z,$ then
$\partial c(z)=\{\nabla c(z)\}$, see for example \cite[Theorem 5.37 (p. 77)]{Tiel84}.

Now for any $x\in\mathbb{R}^{n}$ the function $\xi$ assigns some subgradient, that is $\xi(x)\in\partial
c(x)$.
\begin{definition}
For any point $x\in%
\mathbb{R}
^{n}$,\ the \texttt{subgradient projection}\textit{ }of $x$ is defined as%
\begin{equation}
\Pi_{_{C}}(x):=\left\{
\begin{array}
[c]{ll}%
x-\frac{\displaystyle c(x)}{\displaystyle\left\Vert \xi\right\Vert ^{2}}\xi &
\text{if\ }c(x)>0\text{,}\\
x & \text{if\ }c(x)\leq0\text{,}%
\end{array}
\right.
\end{equation}
where $\xi\in\partial c(x)$. It must be that $\xi \neq 0$ when $c(x)>0$, because if $\xi=0$, then from (\ref{eq:2.3}) one has $c(x)\leq c(y)$ for every $y\in \mathbb{R}^{n}$; in particular, for $y\in C$ we have $c(x)\leq c(y)=0$, a contradiction to the assumption  that $c(x)>0$.
\end{definition}

\begin{remark}
It is well-known and can be verified easily that if the set $C$ is a half-space which is presented using its canonical way (using a normal vector) then the subgradient projection is the orthogonal projection onto $C$.
\end{remark}

\begin{definition}
\label{Def:Slater}Consider the CFP%
\begin{equation}
\boldsymbol{C}:=\cap_{i\in I}C_{i}=\cap_{i\in I}\left\{  x\in%
\mathbb{R}
^{n}\mid g_{i}(x)\leq0\right\}  .
\end{equation}
We say that $\boldsymbol{C}$ satisfies the \texttt{Slater Condition} if there exists a point $x\in C$ having the property that $g_i(x)<0$ for all $i\in I$.
\end{definition}

\subsection{Projection methods and Superiorization}\label{subsec:prog}

Now we present two relevant classes of projection methods, the Orthogonal and
the Subgradient projections methods. We only introduce the
sequential versions which is relevant to our result, but there exists also simultaneous version; see e.g., \cite[Chapter 5]{CZ97}. Later
we also present the Superiorization methodology.\bigskip

\textbf{1. }\underline{\textbf{Sequential methods}}

Sequential projection methods are also refereed to as \textquotedblleft
row-action\textquotedblright\ methods. The main idea is that at each iteration
one constraint set $C_{i}$ is chosen with respect to some control sequence and
either an orthogonal or a subgradient projection is calculated.

\textbf{1.1. }\textit{Projection Onto Convex Sets} (POCS). The general
iterative step can fit into the following%
\begin{equation}
x^{k+1}=x^{k}-\lambda_{k}\left(x^{k}-P_{C_{i(k)}}(x^{k})\right)
\label{eq:POCS}%
\end{equation}
where $\lambda_{k}\in\lbrack\epsilon_{1},2-\epsilon_{2}]$ are called
\textit{relaxation parameters} for arbitrary $\epsilon_{1},\epsilon_{2}>0$ such that $\epsilon_{1}+\epsilon_{2}<2$,
$P_{C_{i(k)}}$ is the orthogonal projection of $x^{k}$ onto $C_{i(k)}$,
$\left\{  i(k)\right\}  $ is a sequence of indices according to which
individual sets $C_{i}$ are chosen, for example \textit{cyclic}
$i(k)=k\operatorname*{mod}m+1$. For the linear case with equalities and
$\lambda_{k}=1$ for all $k$, this is known as Kaczmarz's algorithm
\cite{Kaczmarz37} or \textit{Algebraic Reconstruction Technique} (ART) in the
field of image reconstruction from projection, see \cite{bgh70, hm93}. For
solving a system of interval linear inequalities which appears for example in
the field of \textit{Intensity-Modulated Radiation Therapy} (IMRT), ART3 and
especially its faster version ART3+ (see \cite{hc08}) are known to find a
solution in a finite number of steps, provided that the feasible region is
full dimensional. The successful idea of ART3+ was extended for solving
optimization problems with linear objective and interval linear inequalities
constraints, this is known as ATR3+O \cite{ccmzkh10}.\medskip

\textbf{1.2.} The \textit{Cyclic Subgradient Projections} (CSP) introduced by
Censor and Lent \cite{cl81, cl82} for solving the CFP. The iterative step of
the method is formulated as follows.%
\begin{equation}
x^{k+1}=\left\{
\begin{tabular}
[c]{ll}%
$x^{k}-\lambda_{k}\frac{g_{i(k)}(x^{k})}{\left\Vert \xi^{k}\right\Vert ^{2}%
}\xi^{k}$ & $g_{i(k)}(x^{k})>0$\\
$x^{k}$ & $g_{i(k)}(x^{k})\leq0$%
\end{tabular}
\ \ \ \ \ \ \ \right.\label{csp-def}
\end{equation}
where $\xi^{k}\in\partial g_{i(k)}(x^{k})$ is arbitrary, $\lambda_{k}$ is taken as in
(\ref{eq:POCS}) and $\left\{  i(k)\right\}  $ is cyclic. Of course, in the
linear case this method coincides with POCS.\medskip

\textbf{2.} \underline{\textbf{Superiorization}}

Superiorization is a recently introduced methodology which gains increasing interest and recognition,
as evidenced by the dedicated special issue entitled:
\textquotedblleft Superiorization: Theory and Applications\textquotedblright, in the journal \textit{Inverse Problems} \cite{chj17}.
The state of current research on superiorization can best be appreciated from the \textquotedblleft
Superiorization and Perturbation Resilience of Algorithms: A Bibliography
compiled and continuously updated by Yair Censor\textquotedblright \cite{Censor_sup-page}. In addition, \cite{herman-review-sm}, \cite{weak-strong15} and \cite[Section 4]{rm15} are recent reviews of interest.

This methodology is heuristic and its goal is to find certain good, or superior, solutions to optimization problems. More precisely, suppose that we want to solve a certain optimization problem, for example, minimization of a convex function under constraints (below we focus on this optimization problem because it is relevant to our paper; for an approach which considers the superiorization methodology in a much broader form, see \cite[Section 4]{rm15}). Often, solving the full problem can be rather demanding from the computational point of view, but solving part of it, say the feasibility part (namely, finding a point which satisfies all the constraints) is, in many cases, less demanding. Suppose further that our algorithmic scheme which solves the feasibility problem is perturbation resilient, that is, it converges to a solution of the feasibility problem despite perturbations which may appear in the algorithmic steps due to noise, computational errors, and so on.

Under these assumptions, the superiorization methodology claims that there is an advantage in considering perturbations in an active way during the performance of the scheme which tries to solve the feasibility part. What is this advantage? It may simply be a solution (or an approximation solution) to the feasibility problem which is found faster thanks to the perturbations; it may also be a feasible solution $x'$ which is better than (or superior) feasible solutions $x$ which would have been obtained without the perturbations, where we measure this \textquotedblleft superiority \textquotedblright with respect to some given cost/merit function $\phi$, namely we want to have $\phi(x')\leq \phi(x)$ (and hopefully $\phi(x')$ will be  much smaller than $\phi(x)$).

Since our original optimization problem is the minimization of some convex function, we may, but not obliged to, take $\phi$ to be that function, and we can combine a feasibility-seeking step (a step aiming at finding a solution to the feasibility problem) with a perturbation which will reduce the cost function (such a perturbation can be chosen or be guessed in a non-ascending direction, if such a direction exists:  see Definition \ref{eq:nonascend} and Algorithm \ref{Algorithm:Yair-superiorization} below). We note that the above-mentioned assumption that the algorithmic scheme which solves the feasibility part is perturbation resilient often holds in practice: for example, this is the case for the schemes considered in \cite{bdhk07, cr13, hgcd12}.

\begin{definition}\label{eq:nonascend}
Given a function $\phi:\Delta\subseteq%
\mathbb{R}
^{n}\rightarrow%
\mathbb{R}
$ and a point $x\in\Delta$, we say that a vector $d\in%
\mathbb{R}
^{n}$ is \texttt{non-ascending} \texttt{for }$\phi$\texttt{ at }$x$ if
$\left\Vert d\right\Vert \leq1$ and there is a $\delta>0$ such that%
\begin{equation}
\text{for all }\lambda\in\left[  0,\delta\right]  \text{ we have }\left(
x+\lambda d\right)  \in\Delta\text{ and }\phi\left(  x+\lambda d\right)
\leq\phi\left(  x\right)  . %
\end{equation}
\end{definition}

Observe that one option of choosing the perturbations, in order to steer the algorithm to a superior feasible point with respect to $\phi$, is along $-\nabla\phi$, (when $\phi$ is convex and differentiable) but this is only one example and of course the scheme allows the usage of other direction.

The following pseudocode, which is a small modification of a similar algorithm mentioned in \cite{hgcd12},  illustrates one option to perform the perturbations when applying the superiorizaton methodology.

\begin{algorithm}
\label{Algorithm:Yair-superiorization}$\left.  {}\right.  $

\textbf{Initialization:} Select an arbitrary starting point\textit{ }$x^{0}\in\mathbb{R}^{n}$, a positive integer $N$, an integer $\ell$, a sequence $(\eta_{\ell})_{\ell=0}^\infty$ of positive real numbers which is strictly decreasing to zero (for example, $\eta_{\ell}=a^\ell$ where $a\in(0,1)$) and a family of algorithmic
operators $(P_k)_{k=0}^{\infty}$.

\textbf{Iterative step:}

\textbf{set} $k=0$

\textbf{set} $x^{k}=x^{0}$

\textbf{set} $\ell=-1$

\textbf{repeat until a stopping criterion is satisfied (see Section \ref{sec:GeneralOpt})}

$\qquad$\textbf{set} $m=0$

$\qquad$\textbf{set} $x^{k,m}=x^{k}$

$\qquad$\textbf{while }$m$\textbf{$<$}$N$

$\qquad\qquad$\textbf{set }$v^{k,m}$\textbf{ }to be a non-ascending vector for
$\phi$ at $x^{k,m}$

$\qquad$\textbf{$\qquad$set} \emph{loop=true}

$\qquad$\textbf{$\qquad$while}\emph{ loop}

$\qquad\qquad\qquad$\textbf{set $\ell=\ell+1$}

$\qquad\qquad\qquad$\textbf{set} $\beta_{k,m}=\eta_{\ell}$

$\qquad\qquad\qquad$\textbf{set} $z=x^{k,m}+\beta_{k,m}v^{k,m}$

$\qquad\qquad\qquad$\textbf{if }$z$\textbf{$\in$}$\Delta$\textbf{ and }%
$\phi\left(  z\right)  $\textbf{$\leq$}$\phi\left(  x^{k}\right)  $\textbf{
\textbf{then}}

$\qquad\qquad\qquad\qquad$\textbf{set }$m$\textbf{$=$}$m+1$

$\qquad\qquad\qquad\qquad$\textbf{set }$x^{k,m}$\textbf{$=$}$z$

$\qquad\qquad\qquad\qquad$\textbf{set }\emph{loop = false}

$\qquad$\textbf{set }$x^{k+1}$\textbf{$=$}$\boldsymbol{P}_k(x^{k,N})$

$\qquad$\textbf{set }$k=k+1$
\end{algorithm}

\section{A general projection scheme for convex
optimization\label{sec:GeneralOpt}}

In this section we present our scheme (Algorithm \ref{Algorithm:epsilon-scheme}) for solving convex optimization problems
by translating them into a sequence of convex feasibility problems
(\ref{Problem:k_CFP}) and then solving each of them by using some projection
method; since we are concerned with the general convex case, subgradient
projections are most likely to be used, but we can use any type of projections, for example orthogonal and Bregman projection, see e.g., \cite{CZ97}. There are two essential questions in
this scheme, the first is how to construct the sequence of convex feasibility
problems, meaning how to choose an $\varepsilon$ to update $t_{k}$, and the
other question is when to stop the procedure. For the latter question, that is, the stopping criterion, there are several options and the most poplar are maximum number of iterations (an upper bound on the maximum number should be specified in advance), or to check whether either $\|x^{k+1}-x^k\|$ or $|f(x^{k+1})-f(x^k)|$ are smaller than some given positive parameter.\\

For (\ref{Problem:1}) we denote $C_{i}:=\left\{  x\in%
\mathbb{R}
^{n}\mid g_{i}(x)\leq0\right\}  $, we assume that $\boldsymbol{C}:=\cap_{i\in I}C_{i}%
\neq\emptyset$ and for $t\in%
\mathbb{R}
$ we denote $C^{t}:=\left\{  x\in%
\mathbb{R}
^{n}\mid f(x)\leq t\right\}  $.

We would like to motivate our scheme (Algorithm \ref{Algorithm:epsilon-scheme}) by reviewing a natural extension of the
ART3+O \cite{ccmzkh10}, which was designed for linear problems, for the
general convex case. We will also show how the mathematical disadvantages of
ART3+O can be treated in our new scheme. ART3+O is based on the same
reformulation of the original optimization problem (\ref{Problem:1}) into a
feasibility problem (\ref{Problem:2}). Then the optimal level set value is
determined using a one-dimensional line search on $t_{k}$. In the original
work \cite{ccmzkh10} the \textit{Dichotomous} (bisection) line search
\cite[Chapter 8]{BSS06} was used (in practice a somewhat variant of the line search was used), but any line search could be applied.

Assume that a lower bound of $f$ is given and denote it by $f_{l}$. We denote by
$f^{h}$ the upper bound of $f$ and initialize it to $\infty$.
In \cite{ccmzkh10} there is also the use of a bisection scheme but for the linear case, and in what follows we generalize this scheme for the convex optimization setting (below $k$ is a natural number).
\newpage

\begin{algorithm}
[Bisection scheme]\label{Algorithm:Bisection}$\left.  {}\right.  $

\textbf{Initialization:} Solve the following CFP%
\begin{equation}
\text{find a point }x^{0}\in\boldsymbol{C}%
\end{equation}
set $f^{h}=f(x^{0})$ and $t_{0}=\left(  f_{l}+f^{h}\right)  /2$.\medskip

\textbf{Iterative step:} Given $t_{k-1}$, try to find a feasible solution $x^{k}%
\in\boldsymbol{C}\cap C^{t_{k-1}}$;

(i) If there exists a feasible solution, set $f^{h}=f(x^{k})$ and continue
with $t_{k}:=\left(  f_{l}+f^{h}\right)  /2$;\medskip

(ii) If there is no feasible solution, determined by a \textquotedblleft time-out\textquotedblright\ rule (meaning that a feasible point can
not be found in $n_{max}$ iterations; other alternatives might be \cite{kac91, Kiwiel96} and \cite{cl02}), then set $f_{l}=t_{k-1}$ and
continue with $t_{k}:=\left(  f_{l}+f^{h}\right)  /2$;\medskip

(iii) If $\left\vert f^{h} - f_{l} \right\vert \leq\gamma$ for small enough $\gamma>0$, then stop. A $\gamma$-optimal solution is obtained.
\end{algorithm}

Next we present our new scheme which we call the level set scheme for solving the
constrained minimization (\ref{Problem:1}). Let $\left\{  \varepsilon
_{k}\right\}  _{k=0}^{\infty}$ be some user chosen positive sequence, such
that $\sum_{k=0}^{\infty}\epsilon_{k}=\infty$. We choose $\epsilon_{k}%
=\max\{0.1|f(x^{k})|,0.1\}.$

\begin{algorithm}
[Level set scheme]\label{Algorithm:epsilon-scheme}$\left.  {}\right.  $

\textbf{Initialization:} Solve the following CFP%
\begin{equation}
\text{find a point }x^{0}\in\boldsymbol{C}%
\end{equation}
and set $t_{0}=f(x^{0})-\varepsilon_{0}$.\medskip

\textbf{Iterative step:} Given the current point $x^{k-1}$, try to find a
point $x^{k}\in\boldsymbol{C}\cap C^{t_{k-1}}$;

(i) If there exists a feasible solution, set $t_{k}=f(x^{k})-\varepsilon_{k}$
and continue.\medskip

(ii) If there is no feasible solution, then $x^{k-1}$ is an $\varepsilon_{k}%
$-optimal solution.\medskip
\end{algorithm}

\begin{remark}
Compared with the bisection strategy, infeasibility is detected only once,
just before we get the $\varepsilon_{k}$-optimal solution.\bigskip
\end{remark}

This level set scheme is quite general as it allows users to decide in
advance what projection method they would like to use in that scheme. For the
numerical results, we decided to apply the scheme with the following
variations of projection methods.

\textbf{1.} Each convex feasibility problem in the level set scheme is solved
via the \textit{Cyclic Subgradient Projections Method} (CSPM) (\ref{csp-def}) with $\lambda_k\in(1,2)$ over relaxation parameters.

\textbf{2.} Each convex feasibility problem in the level set scheme is solved
based on the superiorization methodology. By doing so we try to decrease the objective function value below $t_{k-1}$. Following the recent result of
\cite[Section 7]{cr13}, the methodology can be extended to convex and non-convex constraints. In general,
superiorization does not provide an optimality certificate, therefore, we
propose a sequential superiorization method where we decrease the sub-level
sets of the objective function $f$ according to the level set scheme.

\textbf{3.} In the first variation where CSPM is used to solve the resulting
feasibility problems, it may happen that the objective function only decreases
by some small $\varepsilon_k$ in each step $k$. Combining the previous ideas, if only small
steps are detected as progress, a perturbation along the negative gradient of
the objective is performed - just like in superiorization. That is, if
insufficient decrease is detected within a block of iterations, then the current
iterate is shifted by $x^{k}\leftarrow x^{k}-1.9\nabla f(x^{k})$. It is clear that this
is a heuristic step and does not guarantee that $f(x^{k-1}-1.9\nabla f(x^{k-1}))\leq f(x^{k-1})$, and so it can be revised by using an adaptive step-size rule for some positive $\alpha$ such that $f(x^{k-1}-\alpha\nabla f(x^{k-1}))\leq f(x^{k-1})$.

Let $c,s>0$ and $\left\{  \varepsilon_{k}\right\}  _{k=0}^{\infty}$ be a
sequence and small user chosen constants, such that $\sum_{k=0}^{\infty
}\epsilon_{k}=\infty$, and $\delta=0$; In addition, determine the size of each
block of iterations $BLOCK$, for example if we decide to run 1000 iterations
then $BLOCK=1000.$

\begin{algorithm}
\label{Algorithm:modified-epsilon-scheme} $\left.  {}\right.  $

\textbf{Initialization:} Let $\delta=0$ and $t_{-1}:=\infty$; Solve the
following CFP%
\begin{equation}
\text{find a point }x^{0}\in\boldsymbol{C}%
\end{equation}
and set $t_{0}=f(x^{0})-\varepsilon_{0}$.\medskip

\textbf{Iterative step:} At the $k$-th iterate compute $\left\vert
t_{k-2}-t_{k-1}\right\vert $;\medskip

If $\left\vert t_{k-2}-t_{k-1}\right\vert \leq c\varepsilon_{k-1}$, then
$\delta=\delta+1$.\medskip

If $\delta/BLOCK>s,$ then $\delta=0$ and $x^{k-1}\leftarrow x^{k-1}-1.9\nabla
f(x^{k-1}).$\medskip

Set $t_{k-1}=f(x^{k-1})-\varepsilon_{k-1}$ and try to find a solution $x^{k}\in\boldsymbol{C}\cap C^{t_{k-1}}$ to the
CFP.
\end{algorithm}

\begin{remark}
\textbf{Relation with previous work.} There are numerous approaches in
the literature on how to update $t_{k}$, and by that, the sub-level set of $f$ in
solving (\ref{Problem:k_CFP}), or how to transform (\ref{Problem:1}) to a
sequence of CFPs. Some of these schemes are Khabibullin \cite[English
translation]{Khabibullin87_1} and \cite[English translation]{Khabibullin87},
\textit{Cutting-planes methods} or \textit{localization methods}; see
\cite{Kelly60, em75, lnn95, gly96, Kiwiel96, gk99} and \cite{BV07, cl02}, Cegielski in \cite{cegielski99}
and also in \cite{cd02,cd03}, \textit{subgradient method for constrained
optimization}, see e.g. \cite{bm08}. For optimization problem (\ref{Problem:1}%
) with separable objective see De Pierro and Helou Neto \cite{dn09} and see also \cite{ccmzkh10}.
\end{remark}

\subsection{Convergence \label{sec:Convergence}}

Next we present the convergence proof of Algorithm \ref{Algorithm:epsilon-scheme}. We use different arguments
than those presented for other finite convergent projection methods, for example, to name but a few Khabibullin \cite[English translation]{Khabibullin87}, (see also
Kulikov and Fazylov \cite{kf89} and Konnov \cite[Procedure A.]{konnov98}) and
Iusem and Moledo \cite{im87}, De Pierro and Iusem \cite{di98} and Censor, Chen
and Pajoohesh \cite{ccp11}. These algorithms assume that the Slater Condition
holds, i.e., Definition \ref{Def:Slater}.

\begin{theorem}
\label{thm:P_k} Let $(P_{k})_{k=0}^{\infty}$ be a sequence of algorithmic
schemes (formally, each $P_{k}$ is a Turing machine). For every $k\in
\mathbb{N}\cup\{0\}$ the goal of $P_{k}$ is to solve the sub-problem (\ref{Problem:k_CFP}). It
produces, after a finite number of machine operations, an output, and then it
terminates. There are three possible cases for this output: if there exists a
solution to (\ref{Problem:k_CFP}) and the machine is able to find it before it passes a given
threshold (that is, before it performs a too large number of machine
operations, where this \textquotedblleft large number\textquotedblright\hspace{.1in}is fixed in the beginning), then this
output is a point $x^{k}$ which solves (\ref{Problem:k_CFP}); if there exists no solution to
(\ref{Problem:k_CFP}) and the machine is able to determine this case before it passes the
threshold, then the output is a string indicating that (\ref{Problem:k_CFP}) has no solution;
otherwise the output is a string indicating that the threshold has been
passed. In addition, if for some $k\in\mathbb{N}\cup\{0\}$ the algorithmic
scheme $P_{k}$ is able to find a point $x^{k}$ satisfying (\ref{Problem:k_CFP}), then a
positive number $\epsilon_{k}$ is produced (it may or may not depend on
$x^{k}$) and one defines $t_{k}:=f(x^{k})-\epsilon_{k}$. Assume further that
there exists a sequence $(\tilde{\epsilon}_{k})_{k=0}^{\infty}$ satisfying
$\sum_{k=1}^{\infty}\tilde{\epsilon}_{k}=\infty$ and having the property that
for every $k\in\mathbb{N}\cup\{0\}$, if $P_{k}$ is able to find a point
$x^{k}$ satisfying (\ref{Problem:k_CFP}) before passing the threshold, then $\epsilon_{k}%
\geq\tilde{\epsilon}_{k}$.
Under the above mentioned assumptions, Algorithm \ref{Algorithm:epsilon-scheme} terminates after a
finite number of machine operations, and, moreover, exactly one of the
following cases must hold:

\textbf{Case 1:}\label{item:P_0} The only algorithmic scheme that has been applied is
$P_{0}$ and either it declares that (\ref{Problem:k_CFP}) has no solution or it declares that
the threshold has been passed;\medskip

\textbf{Case 2:}\label{item:LastComputed} there exists $k\in\mathbb{N}\cup\{0\}$ such
that $P_{0},\ldots,P_{k}$ are able to solve (\ref{Problem:k_CFP}) before the threshold has
been passed and $P_{k+1}$ terminates by declaring that (\ref{Problem:k_CFP}) does not have a
solution. In this case $x^{k}$ is an $\epsilon_{k}$-approximate solution of
the minimization problem (\ref{Problem:2}).\medskip

\textbf{Case 3:}\label{item:threshold} there exists $k\in\mathbb{N}\cup\{0\}$ such that
$P_{0},\ldots,P_{k}$ are able to solve (\ref{Problem:k_CFP}) before the threshold has been
passed and $P_{k+1}$ terminates by declaring that the threshold has been passed;
\end{theorem}

\begin{proof}
A simple verification shows that the three cases mentioned above are mutually disjoint, and therefore at
most one of them can hold. Hence it is sufficient to show that at least one of
these cases holds. The level-set scheme starts at $k=0$. According to our
assumption on $P_{0}$ (and on any other algorithmic scheme), either it is able
to solve (\ref{Problem:k_CFP}) before passing the threshold, or it is able to show before
passing the threshold that (\ref{Problem:k_CFP}) does not have any solution, or it passes the
threshold before being able to determine whether (\ref{Problem:k_CFP}) has or does not have
any solution. If either the second or the third cases holds, then we are in
the first case (Case 1) mentioned by the theorem, and
the proof is complete (the number of machine operations done on both cases is
finite by the assumption on $P_{0}$). Hence from now on we assume that $P_{0}$
is able to solve (\ref{Problem:k_CFP}) before passing the threshold.

According to the level-set scheme definition, since we assume that $P_{0}$ was
able to solve (\ref{Problem:k_CFP}), we should now consider $P_{1}$. Either $P_{1}$ finds a
solution $x^{1}$ to (\ref{Problem:k_CFP}) before passing the threshold, or it is able to show
before passing the threshold that (\ref{Problem:k_CFP}) does not have any solution, or it
passes the threshold before being able to determine whether (\ref{Problem:k_CFP}) has or does
not have any solution. In the second case we are in Case
2 of the theorem and in the third case we are in Case 3 of the theorem. Hence in the second
and third cases the proof is complete (up to the verification that in the
second case $x^{k}$ is an $\epsilon_{k}$-optimal solution: see the next
paragraph), and so we assume from now on that $P_1$ finds a solution to (\ref{Problem:k_CFP}) before passing the threshold. By continuing this reasoning it can be shown by induction that several subcases
can hold: either any $P_{k}$, $k\in\mathbb{N}\cup\{0\}$, is able to solve
(\ref{Problem:k_CFP}) before passing the threshold, or there exists a minimal $k\in
\mathbb{N}\cup\{0\}$ such that any $P_{j}$, $j\in\{0,\ldots,k\}$ is able to
solve (\ref{Problem:k_CFP}) before passing the threshold but $P_{k+1}$ either shows that (\ref{Problem:k_CFP})
does not have any solution or $P_{k+1}$ passes the threshold before being able
to determine whether (\ref{Problem:k_CFP}) has a solution or does not have any solution. In
the second subcase we are in Case 2 of the
theorem, and in the third subcase we are in Case 3 of the theorem. In both subcases the accumulating machine operations is, of course, finite, since it is the sum of the finitely many machine operations
done by each of the algorithmic schemes $P_{j}$, $j\in\{0,1,\ldots,k+1\}$.

In the third subcase the proof is complete but in the second subcase we also need to
show that $x^{k}$ is an $\epsilon_{k}$-optimal solution. Indeed, suppose that
this subcase holds. Then $C\cap C^{t_{k}}=\emptyset$. Since a basic assumption of
the paper is that the set of minimizers of $f$ over $C$ is non-empty, there
exists $x^{*}\in C$ satisfying $f(x^{*})=t^{*}:=\inf\{f(x): x\in C\}$. It must
be that $t^{*}>t_{k}$ because otherwise we would have $f(x^{*})=t^{*}\leq
t_{k}$, i.e., $x^{*}\in C\cap C^{t_{k}}$, a contradiction. Because $x^{k}\in
C$ one has $t^{*}\leq f(x^{k})$. Hence $t^{*}\leq f(x^{k})=t_{k}+\epsilon
_{k}<t^{*}+\epsilon_{k}$ and therefore $|f(x^{k})-t^{*}|<\epsilon_{k}$. In
other words, $x^{k}$ is an $\epsilon_{k}$-optimal solution, as required.

Therefore it remains to deal with the first subcase mentioned earlier in which
each $P_{k}$, $k\in\mathbb{N}\cup\{0\}$, is able to solve (\ref{Problem:k_CFP}) before passing
the threshold. Assume to the contrary that this subcase holds. Then for
each $k\in\mathbb{N}\cup\{0\}$ the point $x^{k}$ and the numbers $\epsilon
_{k}$ and $t_{k}$ are well-defined and their definitions imply (by induction)
that when $k\geq2$, then
\begin{equation}
t_{k}=f(x^{k})-\epsilon_{k}\leq t_{k-1}-\epsilon_{k}\leq\ldots\leq t_{0}%
-\sum_{j=1}^{k}\epsilon_{j}\leq t_{0}-\sum_{j=1}^{k}\tilde{\epsilon}_{j}.
\label{t_k<t_0-sum}%
\end{equation}
Because $t^{\ast}=f(x^{\ast})\in\mathbb{R}$ and since, according to our
assumption, $\sum_{j=1}^{\infty}\tilde{\epsilon}_{j}=\infty$, for large enough
$k\in\mathbb{N}$ we have $t_{0}-t^{\ast}<\sum_{j=1}^{k}\tilde{\epsilon}_{j}$.
By combining this with (\ref{t_k<t_0-sum}) it follows that
\begin{equation}
t_{k}\leq t_{0}-\sum_{j=1}^{k}\tilde{\epsilon}_{j}<t^{\ast}. \label{t_k<v_0}%
\end{equation}
It must be that one of the algorithmic schemes $P_{j}$, $j\in\{1,\ldots,k+1\}$
will fail to solve (\ref{Problem:k_CFP}) (either by passing the threshold or by
determining that (\ref{Problem:k_CFP}) does not have any solution), since if this is not true,
then in iteration $k+1$ a solution $x^{k+1}$ to (\ref{Problem:k_CFP}) will be found by
$P_{k+1}$. Now, because $x^{k+1}$ solves (\ref{Problem:k_CFP}) we have $f(x^{k+1})\leq t_{k}$.
Hence it follows from (\ref{t_k<v_0}) that $f(x^{k+1})<t^{\ast}$, a
contradiction to the definition of $t^{\ast}$. This contradiction shows that
the subcase mentioned earlier in which each $P_k$, $k\in \mathbb{N}\cup\{0\}$ is able to solve (\ref{Problem:k_CFP}) before passing the threshold, cannot occur.\bigskip
\end{proof}

\begin{remark}
If for some given $\epsilon>0$ the sequence $\{\epsilon_{k}\}_{k=0}^{\infty}$
satisfies $\epsilon_{k}<\epsilon$ for all $k\in\mathbb{N}$ sufficiently large,
then the theorem ensures, in the third case mentioned in it, that the point
$x^{k}$ will be an $\epsilon$-approximate solution.
\end{remark}

\begin{remark}
An illustration of the condition needed in Theorem \ref{thm:P_k} is to let $\epsilon
_{k}:=\max\{0.1,0.1 |f(x^{k})|\}$, as done in the numerical simulations
(Section \ref{sec:Numerical_Experiments}). In this case $\epsilon_{k}\geq\tilde{\epsilon}_{k}:=0.1$ for all
$k\in\mathbb{N}\cup\{0\}$ for which $P_{k}$ is able to solve (\ref{Problem:k_CFP}) before
passing the threshold, and, in addition, $\sum_{k=0}^{\infty} \tilde{\epsilon
}_{k}=\infty$, as required. However, if one merely defines $\epsilon_{k}:=0.1
|f(x^{k})|$ instead of defining $\epsilon_{k}:=\max\{0.1,0.1 |f(x^{k})|\}$,
or, more generally, if one uses algorithmic schemes $P_{k}$ which, for every
$k\in\mathbb{N}\cup\{0\}$, are able to solve (\ref{Problem:k_CFP}) before passing the
threshold, and if $\sum_{k=0}^{\infty}\epsilon_{k}<\infty$, then it may happen
that none of the values $f(x^{k})$ approximate well the optimal value $t^{*}$.

Indeed, consider $f(x):=x^{2}-100$, $x\in C:=\mathbb{R}$. Denote $x^{0}:=\sqrt{500}$.
Suppose that for each $k\in\mathbb{N}$ our schemes $P_{k}$ find a point
$x^{k}$ satisfying $f(x^{k})=t_{k-1}$, namely $x^{k}=\sqrt{100+t_{k-1}}$ (we
assume that $P_{k}$ can represent numbers in an algebraic way which allows it
to store square roots without the need to represent them in a decimal way;
such schemes can be found in the scientific domains called \textquotedblleft
computer algebra\textquotedblright, \textquotedblleft exact numerical
computation\textquotedblright, and \textquotedblleft symbolic
computation\textquotedblright). Let $\epsilon_{0}:=0.1|f(x^{0})|$. Since
$f(x^{0})=400>0$ we have $t_{0}=f(x^{0})-\epsilon_{0}=0.9\cdot400=360$. Now we
need to find a point $x^{1}$ satisfying $f(x^{1})=360$, i.e.,
\[
(x^{1})^{2}-100=360=0.9\cdot f(x^{0})=0.9((x^{0})^{2}-100)=0.9(x^{0})^{2}-90
\]
and thus $x^{1}=\sqrt{10+0.9(x^{0})^{2}}$. By induction
$x^{k}=\sqrt{10+0.9(x^{k-1})^{2}}$, $f(x^{k})=t_{k-1}$ and $\epsilon
_{k}:=0.1|f(x^{k})|$ for every $k\in\mathbb{N}$. In particular, we can see by
induction that $x^{k}\geq10$ for all $k$ and hence $\epsilon_{k}=0.1f(x^{k})$
and $t_{k}=f(x^k)-\epsilon_k=0.9f(x^{k})\geq0$ for all $k\in\mathbb{N}$. Therefore
$|t_{k}-t^{\ast}|\geq100$ and $|f(x^{k})-t^{\ast}|>100$ for all $k\in
\mathbb{N}$ and hence we neither have $\lim_{k\rightarrow\infty}%
f(x^{k})=-100=t^{\ast}$ nor $\lim_{k\rightarrow\infty}t_{k}=t^{\ast}$. It
remains to show that $\sum_{k=1}^{\infty}\epsilon_{k}<\infty$. Indeed, observe
that since $0<f(x^{k+1})\leq t_{k}$ for every $k\in\mathbb{N}\cup\{0\}$ and
since from (\ref{t_k<t_0-sum}) we have $\sum_{i=0}^{k}\epsilon
_{i}\leq t_{0}-t_{k}\leq t_{0}$, it follows that $\sum_{k=1}^{\infty}%
\epsilon_{k}\leq t_{0}$. Therefore $\sum_{k=1}^{\infty}\epsilon
_{k}<\infty$ as claimed.\bigskip
\end{remark}

\section{Numerical experiments \label{sec:Numerical_Experiments}}

In this section, we compare several variants of the two optimization schemes (Algorithms \ref{Algorithm:epsilon-scheme}
and \ref{Algorithm:modified-epsilon-scheme}) for some selected optimization problems. All solvers were tested against the
freely available library of convex quadratic programming problems stored in
the QPS format by Maros and M\'{e}sz\'{a}ros \cite{mm99} as well as clinical
cases from intensity modulated radiation therapy planning (IMRT) provided to
us by the German Cancer Research Center (DKFZ) in Heidelberg. The QPS problems
were parsed using the parser from the CoinUtils package \cite{CoinUtils} and consist of quadratic
objectives and linear constraints. The IMRT problem data is constructed using
a prototypical treatment planning system developed by the Fraunhofer ITWM and
consist of nonlinear convex objectives and constraints.

\begin{remark}
\label{remark:conv} It is clear that from the mathematical point of view only
finite convergence projection methods can be applied in each iterative step.
However, numerical experiments show that even asymptotically convergent
algorithms can be used, when the stopping rule is chosen in an educated way.
For further finite convergence methods see \cite{pm79, fuku82, mph81}.
\end{remark}

The algorithms were implemented in C++. As solvers for the feasibility
problem, we implemented the finite convergence variants of the Cyclic Subgradient Projections Method (CSPM) and
the Algebraic Reconstruction Technique 3 (ART3+) and also their regular version with standard stopping rule, see the beginning of Section \ref{sec:GeneralOpt} and Remark \ref{remark:conv}. The superiorized versions of these methods simply use the objective function of the optimization problem as
a merit function to decrease. Although the superiorized versions of CSPM and
ART3+ preformed surprisingly well in terms of the objective function value
they obtained, the solutions were in almost all cases far from the optimum.
Table \ref{tab:solver_variants} lists the variants of the level set and bisection
schemes that are compared:

\medskip\begin{table}[htp]%
\begin{tabular}
[c]{|l|l|}\hline
Scheme variant & Abbreviation\\\hline
Level set (Alg. \ref{Algorithm:epsilon-scheme}) with CSPM & ls\_cspm\\\hline
Level set (Alg. \ref{Algorithm:epsilon-scheme}) with ART3+ & ls\_art3+\\\hline
Accelerated level set (Alg. \ref{Algorithm:modified-epsilon-scheme}) with
CSPM & ls\_acc\_cspm\\\hline
Level set (Alg. \ref{Algorithm:epsilon-scheme}) with superiorized CSPM &
ls\_sup\_cspm\\\hline
Level set (Alg. \ref{Algorithm:epsilon-scheme}) with superiorized ART3+ &
ls\_sup\_art3+\\\hline
Accelerated level set (Alg. \ref{Algorithm:modified-epsilon-scheme}) with
superiorized CSPM & ls\_acc\_sup\_cspm\\\hline
Bisection (Alg. \ref{Algorithm:Bisection}) with CSPM & bis\_cspm\\\hline
Bisection (Alg. \ref{Algorithm:Bisection}) with ART3+ & bis\_art3+\\\hline
Accelerated Bisection with CSPM & bis\_acc\_cspm\\\hline
Bisection (Alg. \ref{Algorithm:Bisection}) with superiorized CSPM &
bis\_sup\_cspm\\\hline
Bisection (Alg. \ref{Algorithm:Bisection}) with superiorized ART3+ &
bis\_sup\_art3+\\\hline
Accelerated Bisection with superiorized CSPM & bis\_acc\_sup\_cspm\\\hline
\end{tabular}
\caption{Overview of all tested schemes.}%
\label{tab:solver_variants}%
\end{table}\medskip

The bisection schemes were accelerated in the same way as in Algorithm
\ref{Algorithm:modified-epsilon-scheme}, despite the fact that this is a heuristics which is not guaranteed to converge, we decided to test and add it to our comparison. To determine whether a feasible solution exists, we set a maximum number of 1000 iterations for each of the
feasibility solvers. In Algorithm \ref{Algorithm:Bisection}(iii) we choose $\gamma=10^{-5}$. If no feasible solution is found after 1000 projections,
it is assumed that none exists. For the choice of $\varepsilon_{k}$ we used
multiplicative update rule ($\varepsilon_{k}=0.1|f(x^{k})|$) if the absolute value of the objective function is
greater than $1$ and subtraction update rule ($\varepsilon_{k}=0.1$) otherwise.

\subsection{IMRT cases descriptions}

Given fixed irradiation directions, the objective in IMRT optimization is to
determine a treatment plan consisting of an energy fluence distribution to
create a dose in the patient to irradiate the tumor as homoegeneously as
possible while sparing critical healthy organs \cite{KueMonSche09}. Figure
\ref{fig:imrt} shows irradiation directions and some energy fluence maps for a paraspinal tumor case.

\begin{figure}[H]
\centering
\includegraphics[width=0.7\textwidth,height=0.6\textheight,keepaspectratio]
{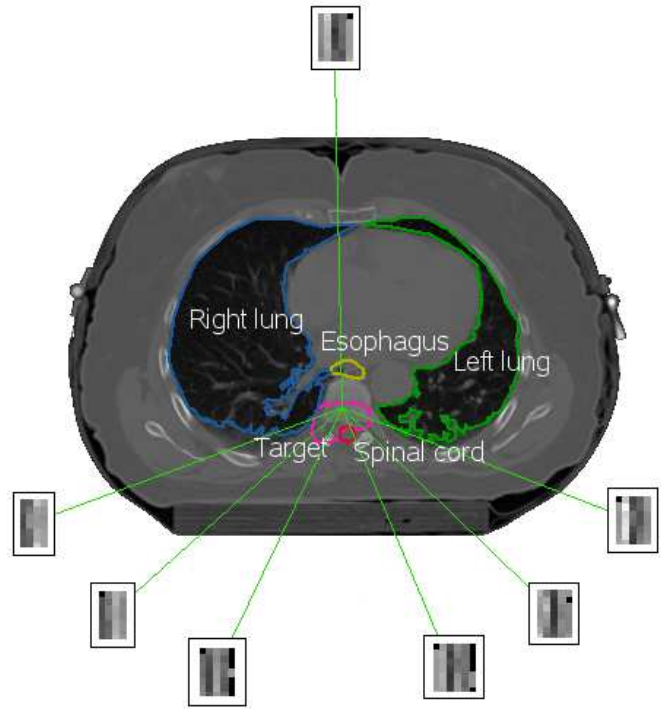}
\caption{The gantry moves around the couch on which the patient lies. The couch position
may also be changed to alter the beam directions.}%
\label{fig:imrt}%
\end{figure}

The problem is multi-criteria in nature and numerical optimization
problems are often a weighted sum scalarization of the multiple
objectives involved. IMRT optimization problems can be formulated
so that they are convex. For this work, we selected nine
head-and-neck cancer patients and posed the same optimization
formulations for each to ensure comparability, see Figure
\ref{fig:Head-Neck} for two of the nine patients. We then chose
four types of scalarization weights to determine four treatment
plans for each patient, each with different distinct solution
properties. Overall, this resulted in 36 optimization problems.
The following list describes the different scalarizations.\\

\textbf{1. }High weights on tumor volumes, low weights on healthy organs;

\textbf{2. }High weights on tumor volumes and brain stem;

\textbf{3. }High weights on tumor volumes and spinal cord;

\textbf{4. }High weights on tumor volumes and parotis glands.\\

In order to numerically optimize the treatment plans, we define
the energy fluence distribution - the variables of the optimization problem - as a vector $x\in\mathbb{R}^{n}$ (typically $n\approx 10^3$)
and assume that the resulting radiation dose in the patient is
given by $d=Dx\in\mathbb{R}^{m}$ (typically $m\approx 10^6$), where the entries $D_{ij}$ of
the so-called dose matrix $D\in\mathbb{R}^{m\times n}$ contain the
information of how much radiation is deposited in voxel $i$ of the
patient body by a unit amount of energy emitted by a small area
$j$ on the beam surface. That is, the dose in each voxel $i$ is given by
$ d_i(x):= \sum_{j=1}^n D_{ij} x_j $. To achieve a homogeneous dose in a tumor
volume given by voxel indices $\mathcal{T}$, we use functions to
minimize the amount of under-dosage below a prescribed dose $R$,
\[
f_{\text{under}}(x):=\left(  \left\vert \mathcal{T}\right\vert ^{-1}\sum
_{i\in\mathcal{T}}\max(0,R-d_{i}(x))^{2}\right)  ^{\frac{1}{2}},
\]
and, symmetrically the over-dosage above a given prescription. This is done
using two functions to provide better control over both aspects. These
objectives are also constrained from above, resulting in nonlinear but convex
constraints. Dose in risk organs given by voxel indices $\mathcal{R}$ is
minimized by norms of the dose in those organs:
\[
f_{\text{norm}}(x):=\left(  \left\vert \mathcal{R}\right\vert ^{-1}\sum
_{i\in\mathcal{R}}d_{i}^{p}(x)\right)  ^{\frac{1}{p}},
\]
where we used $p=2$ and $p=8$, depending on whether the organ is more
sensitive to the general amount of radiation (e.g. parotis glands) or the
maximal dose (e.g. spinal cord) received. More relevant data which is typical and standard to IMRT and in particular for the implementation of our scheme, i.e., the constrains set $\boldsymbol{C}$, can be found in \cite{KueMonSche09}, which also includes details on numerical optimization in IMRT planning; see also the works \cite{cekb05, aygk1505}. Note that the IMRT problem formulations here do not contain any linear constraints in our case, so that in the analysis, ART3+ is
omitted as it is identical to CSPM in this case.

\begin{figure}[htp]
\centering \subfigure[Head-Neck patient 1]{
\includegraphics[width=0.6\textwidth,height=0.23\textheight,keepaspectratio]{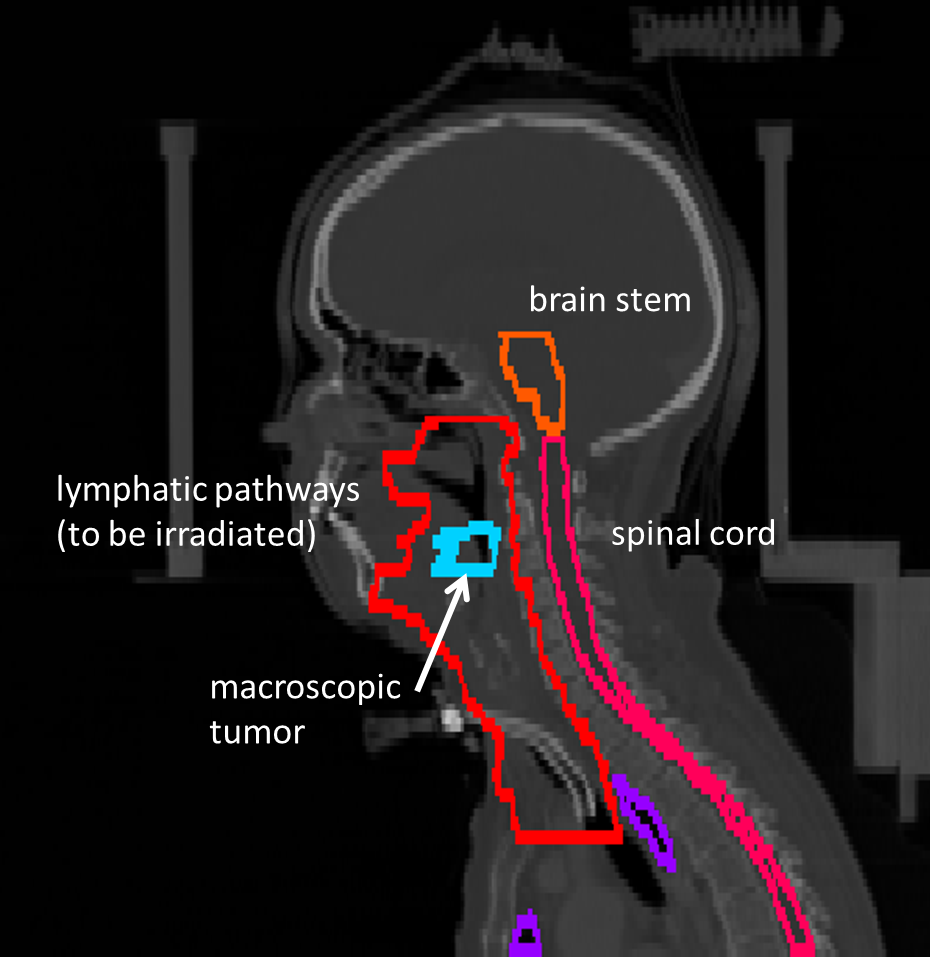}
\label{fig:imrt_case_hnaa_sagital}
\includegraphics[width=0.6\textwidth,height=0.23\textheight,keepaspectratio]{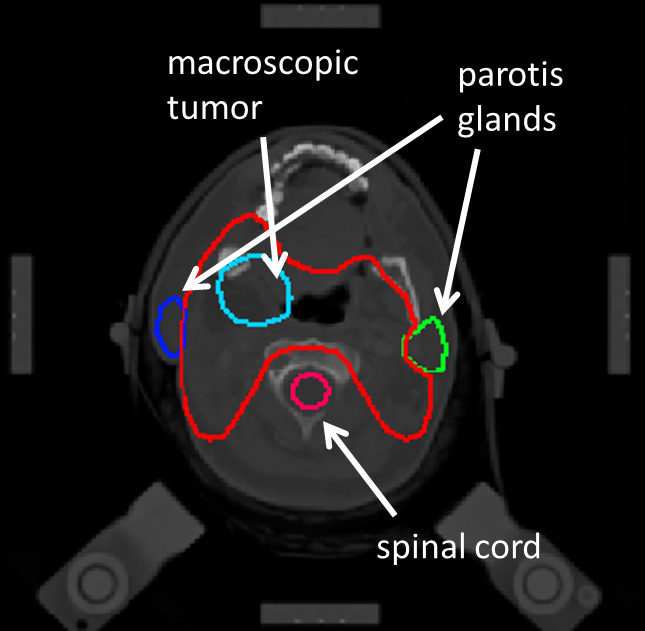}
\label{fig:imrt_case_hnaa_trans} } \subfigure[Head-Neck patient 2]{
\includegraphics[width=0.6\textwidth,height=0.23\textheight,keepaspectratio]{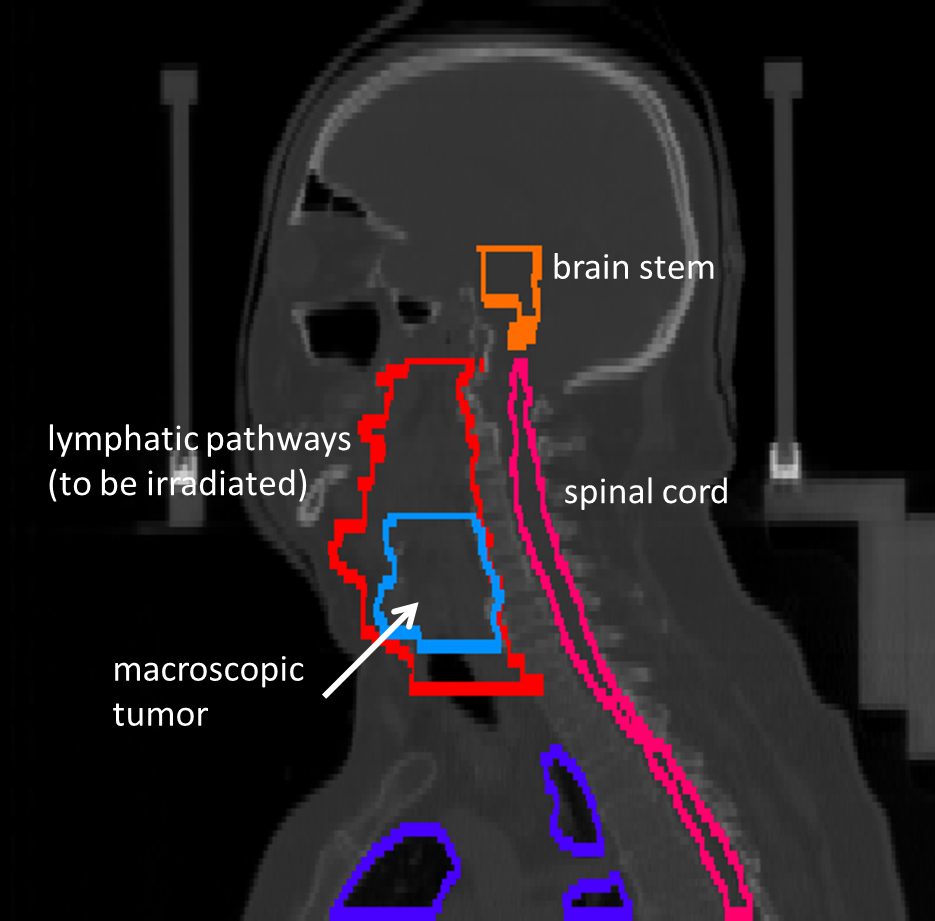}
\label{fig:imrt_case_hnbb_sagital}
\includegraphics[width=0.6\textwidth,height=0.23\textheight,keepaspectratio]{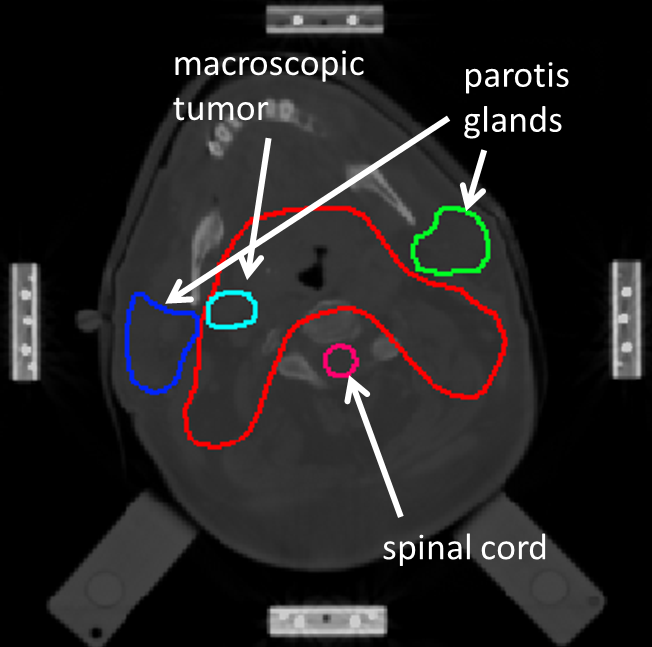}
\label{fig:imrt_case_hnbb_trans} }\caption{Two of the nine
patients. Highlighted are some tumor volumes (lymphatic pathways
and the macroscopic tumor volume) and some healthy
structures to be spared.}%
\label{fig:Head-Neck}%
\end{figure}
\subsection{Quality evaluation of the solutions}

All of the 36 IMRT optimization problems could be solved by all variants of
the schemes. This was not the case for the QPS problems: of the 101 problems
tested in the library, only 50 problems could be solved by all of the variants of the
schemes. There are two reasons why for the other 51 cases, the projection
methods were unable to find an initial feasible solution. The first reason is that most of these QPS problems are indeed infeasible. The second reason is, that in \cite{mm99}, the primal infeasibility stopping criterion is determined as $\|Ax-b\|/(1-\|b\|)<10^{-8}$ (where $A$ and $b$ are part of the QPS problems constraints and $\|b\|<1$) (also combined with an additional stopping criterion for the dual infeasibility), while in our implementations, we choose the maximum number of iterations, such as in \cite{ccmzkh10}, denoted there by $Q$, to be the stopping rule. As can be seen, these two stopping criterion are different. It turn out that even when the number of iterations was increased to $10,000$, still the CSPM did not make a difference and hence these problems were declared infeasible.
In general, the behaviour of projection methods in the inconsistent case (infeasibility) have attracted many researchers and the subject is not fully explored. Some of the results in this area, state that in the case of infeasibility there is a cyclic convergence while for others methods, mainly simultaneous ones, there is convergent to some point that minimizes the norm of the infeasibilities. For further details on the above, the readers are refereed to the works Gubin, Polyak and Raik\'{\i}s \cite[Theorem 2]{gpr67}, Censor and Tom \cite {ct03}, the book of Chinneck \cite{Chinneck08} and the many references therein. Moreover, in the recent result of Censor and Zur \cite{cz16}, superiorization is used for the inconsistent linear case and it is shown that the generated sequence converges to a point that minimizes a proximity function which measures the linear constraints violation.

The following analysis concerning the QPS problems is restricted to those problems that could be solved by all variants. We report the required total number of projections and the total number of objective evaluations for each method as measures of numerical complexity, since these are independent of machine architecture or efficiency of implementation (parallelization or other software acceleration techniques). To measure the quality of the solutions, the following score $Q$ was calculated for each solver variant and each problem. Let $\hat{f}$ be the best objective the solver found and $f^{*} $ the best known objective value for the problem. We define
\[
Q :=
\begin{cases}
\hat{f}, & \mbox{if $f^*=0$}\\
\hat{f} - f^{*}, & \mbox{if $|f^*| \leq 1$}\\
(\hat{f} - f^{*}) / (|f^{*}|), & \mbox{else}
\end{cases}
\]
Thus, the close to 0, the better the score, a positive value for $Q $ measures
a deviation from optimality. Tables
\ref{table:SummaryQualityScoresAlgorithms_QPS} and
\ref{table:SummaryQualityScoresAlgorithms_IMRT} show some statistics for the
deviation from optimality for the solvers.

The median quality score of all optimization schemes for the problems that
could be solved are very good, meaning each solver can be expected to find the
optimal solution if the underlying projection method can find a feasible
starting point. The average is heavily skewed towards some outliers, i.e. instances for
which the algorithms needed many iterations -
especially for the QPS problems and in the bisection schemes for both problem
types. However, not all problems are solved very well, as the 90-th quantiles show: in 10\%
of all problem instances the algorithms did not produce a very good solution.

Based on these findings, the following questions are answered in the following subsections:

\textbf{1. }Do the accelerated versions of the schemes outperform the basic
versions in terms of quality and complexity?

\textbf{2. }Do the superiorized versions of the feasibility solvers outperform
their basic variants since the average deviations are lower for those solvers
with \textquotedblleft sup\textquotedblright?

\textbf{3. }Is the level set scheme better than the bisection scheme in terms
of quality and complexity?

\begin{table}[H]%
\begin{tabular}
[c]{|l|c|c|c|c|}\hline
Scheme variant & Average Q & Median Q & 10-th quantile Q & 90-th quantile
Q\\\hline
ls\_cspm & 0.44 & 0.05 & 0.00 & 0.66\\\hline
ls\_art3+ & 1.83 & 0.05 & 0.00 & 0.66\\\hline
ls\_acc\_cspm & 0.44 & 0.06 & 0.00 & 0.66\\\hline
ls\_sup\_cspm & 0.17 & 0.05 & 0.00 & 0.66\\\hline
ls\_sup\_art3+ & 0.14 & 0.06 & 0.00 & 0.66\\\hline
ls\_acc\_sup\_cspm & 0.20 & 0.05 & 0.00 & 0.66\\\hline
bis\_cspm & 7.22 & 0.04 & 0.00 & 1.10\\\hline
bis\_art3+ & 7.07 & 0.04 & 0.00 & 1.00\\\hline
bis\_acc\_cspm & 7.24 & 0.05 & 0.00 & 1.10\\\hline
bis\_sup\_cspm & 0.19 & 0.03 & 0.00 & 0.89\\\hline
bis\_sup\_art3+ & 0.16 & 0.02 & 0.00 & 0.83\\\hline
bis\_acc\_sup\_cspm & 0.21 & 0.05 & 0.00 & 0.89\\\hline
\end{tabular}
\caption{Quality scores of all tested algorithms over all 50 QPS problems that
could be solved by all variants. The statistics are taken over the 50 problem
runs, thus aggregating the results over all problems.}%
\label{table:SummaryQualityScoresAlgorithms_QPS}%
\end{table}

\begin{table}[H]%
\begin{tabular}
[c]{|l|c|c|c|c|}\hline
Scheme variant & Average Q & Median Q & 10-th quantile Q & 90-th quantile
Q\\\hline
ls\_cspm & 0.13 & 0.11 & 0.03 & 0.23\\\hline
ls\_acc\_cspm & 0.13 & 0.11 & 0.03 & 0.23\\\hline
ls\_sup\_cspm & 0.06 & 0.06 & 0.00 & 0.12\\\hline
ls\_acc\_sup\_cspm & 0.05 & 0.05 & 0.00 & 0.13\\\hline
bis\_cspm & 2.72 & 0.02 & 0.00 & 8.52\\\hline
bis\_acc\_cspm & 2.72 & 0.02 & 0.00 & 8.52\\\hline
bis\_sup\_cspm & 0.04 & 0.04 & 0.00 & 0.08\\\hline
bis\_acc\_sup\_cspm & 0.04 & 0.03 & 0.00 & 0.09\\\hline
\end{tabular}
\caption{Quality scores of all tested algorithms over all IMRT problems. ART3+
was omitted as the problem formulations did not contain any linear
constraints. The statistics are taken over the 36 problem runs, thus
aggregating the results over all problems.}%
\label{table:SummaryQualityScoresAlgorithms_IMRT}%
\end{table}

\subsection{Is the accelerated scheme better than the basic scheme?}

Only the CSPM variants for the level set scheme and the bisection scheme are
studied, since they are most promising candidates for each (in the bisection
scheme, there is no difference between CSPM and ART3+). The quality scores for
the level set scheme with CSPM ls\_cspm and the accelerated level set scheme
ls\_acc\_cspm were compared to see if there is a statistically significant
difference in the outcome of the methods. For the 50 QPS problems, the median
difference is 0, and the t-Test for two-tailed sample mean difference returns
a p-value of 0.472, indicating that there is no real difference between the
two versions. For the IMRT problems there was absolutely no difference in the
quality score between the two variants.

The results for the bisection scheme are even clearer. For the QPS problems,
the median difference is 0 and the average difference is -0.01. In no case
could the accelerated version produce a better objective score, and at worst,
it produced a loss of 0.28 objective score. A statistical test was not
performed for this case. Similar results were found for the IMRT problems.

As there is no difference in quality, the question remains whether the accelerated variants can be better in terms of  faster function decrease or number of constraint projections or function evaluations. For IMRT, there was no difference in running time or rate of decrease of the objective function: the behavior was identical for the base
and accelerated versions. For the QPS problems, however, the acceleration of
the level set schemes lead to a faster converging algorithm. The rate of
decrease of the objective function over all feasible solutions produced by the
scheme with the accelerated version can be up to 15 times the rate of the
basic version. The chart in Figure \ref{fig:speedupObjDecLsAccCspmVsLsCspm}
shows the frequencies of different speedup factors realized for the
accelerated level set scheme over the basic level set scheme. The bins on the
horizontal axis denote the multiplication factor of how much faster the
objective scores (in case of Figure
\ref{fig:speedupObjDecLsAccCspmVsLsCspmQPS}) decreased in the accelerated
cases. ``Frequency'' refers to the count of solver instances where the
multiplication factor was observed. It should be noted that the frequency for
the 0 bin is exactly the count of zero speedups, meaning that there were no cases where the accelerated version
was better in terms of faster function decrease.

\begin{figure}[H]
\centering
\subfigure[Level set scheme]{
\includegraphics[width=0.46\textwidth,height=0.3\textheight,keepaspectratio]{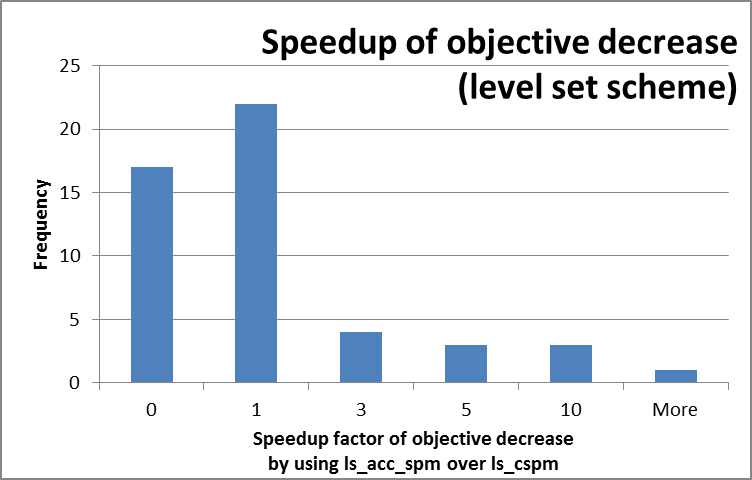}
\label{fig:speedupObjDecLsAccCspmVsLsCspm}
} \subfigure[Bisection scheme]{
\includegraphics[width=0.46\textwidth,height=0.3\textheight,keepaspectratio]{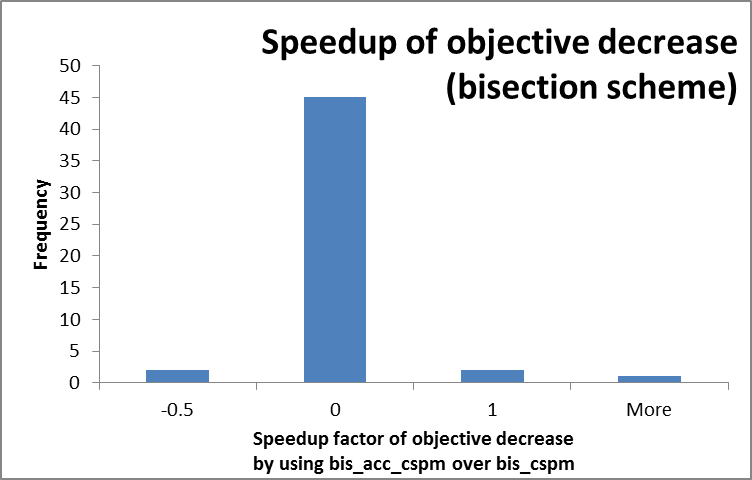}
}\caption{Factor of how much the objective function decreased faster in the
accelerated versions of the schemes (QPS problems). A speedup factor of 0 means that no difference in speed was observed. Negative speedups indicate that the
indicated scheme performed worse over the compared method.}%
\label{fig:speedupObjDecLsAccCspmVsLsCspmQPS}%
\end{figure}

For the bisection scheme and QPS problems, there is no difference in the rate
of objective decrease between accelerated and basic version. Therefore, for
certain problem types, the accelerated level set scheme using CSPM has a great
potential to speed up the rate of objective decrease and only causes a very
moderate increase in the number of projections required by the algorithm (only
about 1.5 times as many for 17 problems).

\subsection{Are the superiorized versions a good choice for the level set
scheme?}

The quality scores in Tables \ref{table:SummaryQualityScoresAlgorithms_QPS}
and \ref{table:SummaryQualityScoresAlgorithms_IMRT} seem to indicate that the
superiorized version outperform their basic variants in both schemes. However,
for the QPS problems, the t-tests for paired two sample mean comparisons
showed that there is, in fact, no statistically significant difference in the
means (p-values for two-tailed tests were 0.148 for the level set scheme
comparison between ls\_acc\_cspm and ls\_acc\_sup\_cspm and 0.292 for the
bisection scheme comparison between bis\_art3+ and bis\_sup\_art3+).
Nevertheless, for our test cases, the superiorized versions in the level set
scheme outperformed their basic counterparts in more cases: the superiorized
version ls\_sup\_acc\_cspm produced a quality score at least as good as the
basic version ls\_acc\_cspm in about 68\% of all cases, in 64\% it could
actually get a better score.

On the other hand, for the IMRT problems, the superiorized versions clearly
outperform the basic versions in terms of objective function scores.
Statistical tests are all significant up to a level of 0.0013 (p-value for
two-tailed tests of test for mean difference being 0). This clearly shows a
promising feature of superiorized algorithms: they are very often able to
obtain better solutions, even when used in an optimization framework. The
accelerated versions of the superiorized schemes, however, did not differ in
quality from their unaccelerated versions.

However, the superiorized versions require more projections and objective
evaluations than the normal versions (see Figures
\ref{fig:SpeedupCompLsSupAccCspmVsLsAccCspmQPS} and
\ref{fig:SpeedupCompLsSupCspmVsLsCspmIMRT}). In these charts, the horizontal
axis again denotes the multiplication factor by which the number of
projections or objective function evaluations of the basic versions would have
to be multiplied to be equal to the values for the superiorized versions. An increase factor of 0 means that the indicated scheme needed as many projections as the compared method. In
many instances, the superiorized versions required more than 10 times the
number of projections over the basic variant, leading to significantly higher
computation times.

\begin{figure}[htp]
\centering \subfigure[Increase in projections]{
\includegraphics[width=0.45\textwidth,height=0.3\textheight,keepaspectratio]{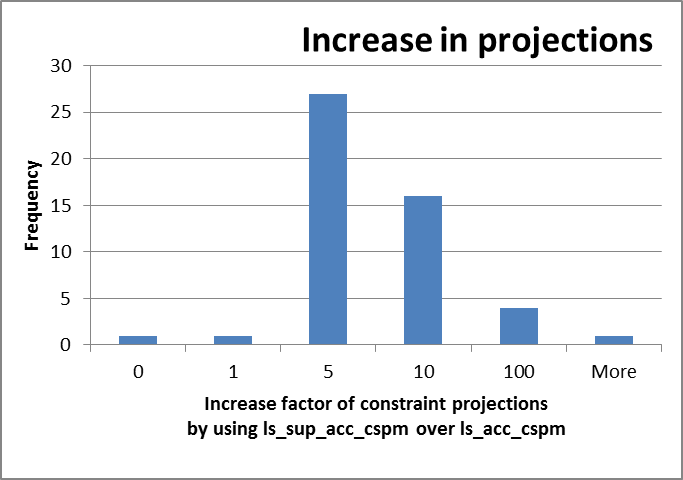}
} \subfigure[Increase in objective evaluations]{
\includegraphics[width=0.45\textwidth,height=0.3\textheight,keepaspectratio]{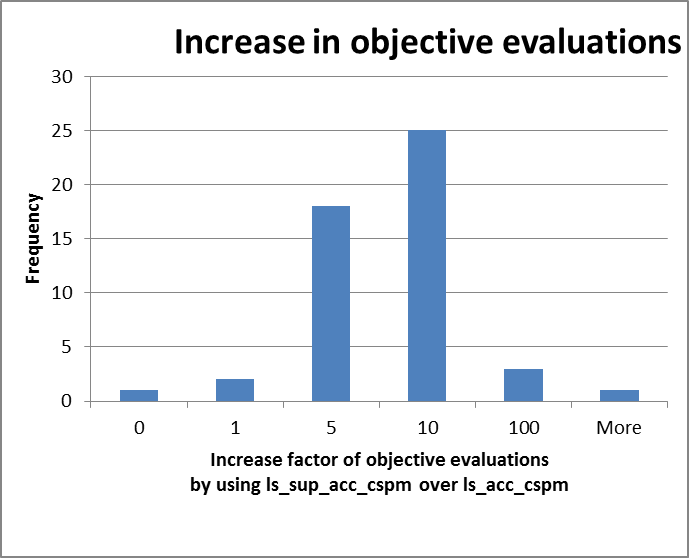}
}\caption{Factor of how much the complexity increases by using the
superiorized version of CSPM in the level set scheme (QPS problems). }%
\label{fig:SpeedupCompLsSupAccCspmVsLsAccCspmQPS}%
\end{figure}

\begin{figure}[htp]
\centering
\subfigure[Increase in projections]{
\includegraphics[width=0.45\textwidth]{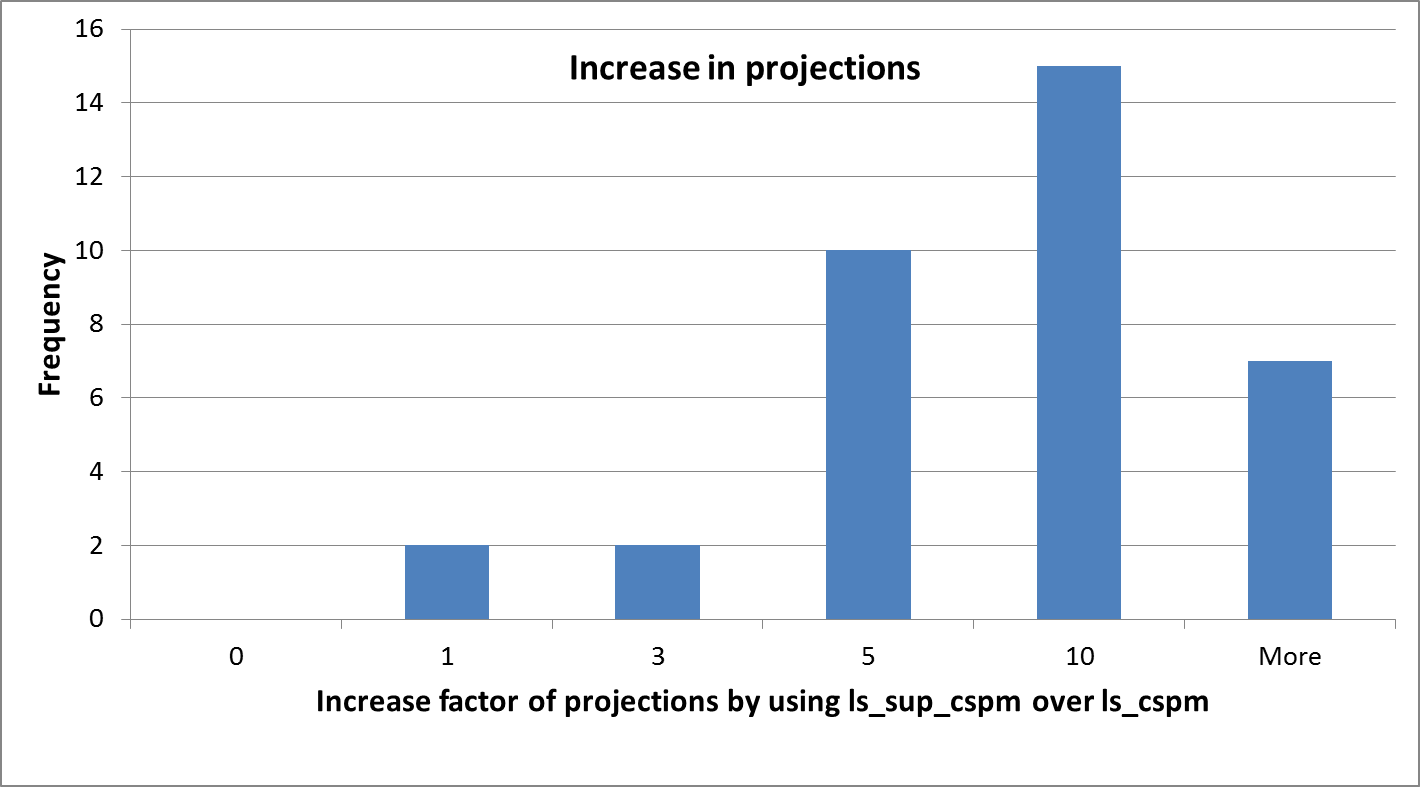}
} \subfigure[Increase in objective evaluations]{
\includegraphics[width=0.45\textwidth]{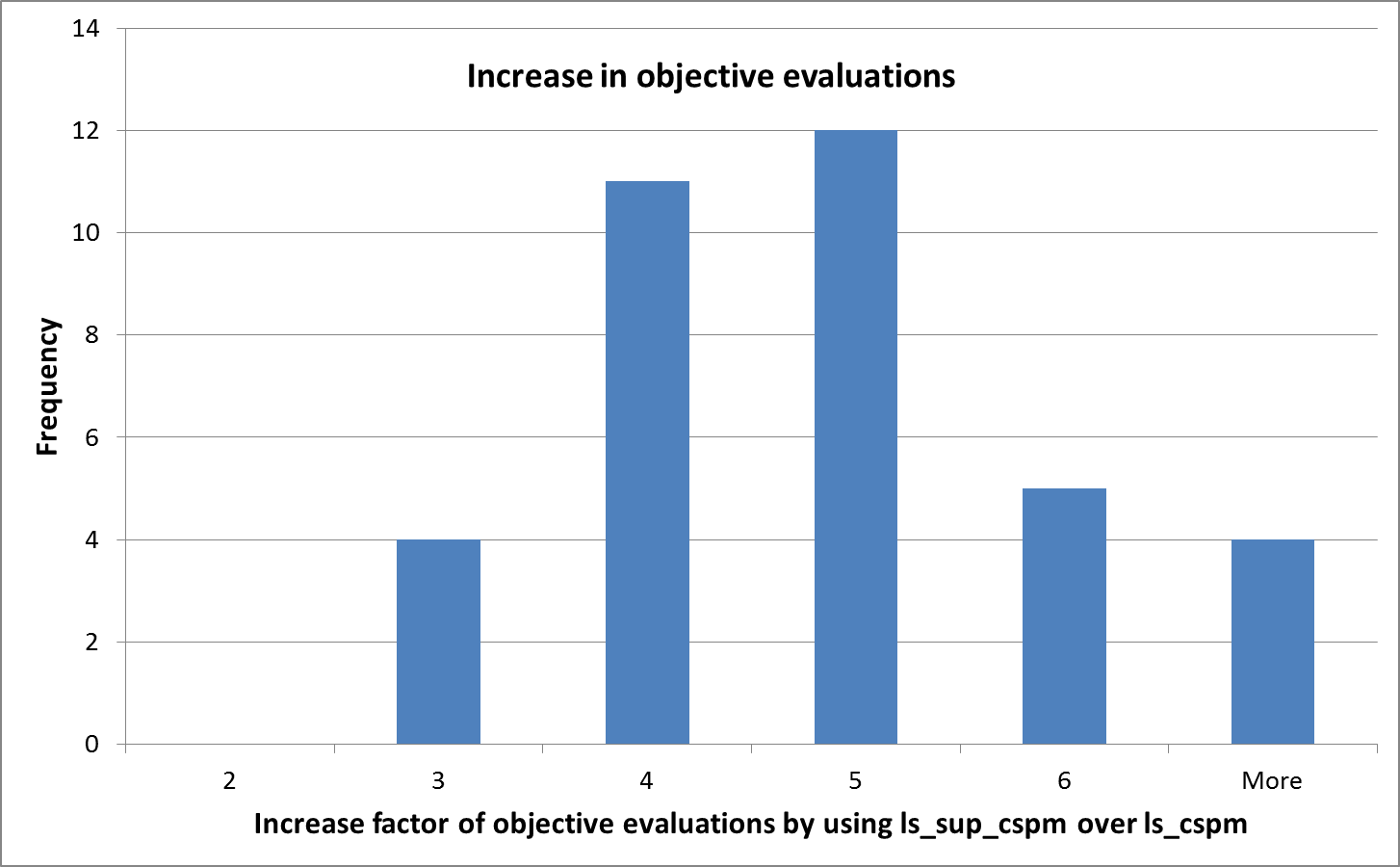}
}\caption{Factor of how much the complexity increases by using the
superiorized version of CSPM in the level set scheme (IMRT problems).}%
\label{fig:SpeedupCompLsSupCspmVsLsCspmIMRT}%
\end{figure}

Yet, for the QPS problems, there is also some potential when it comes to the
rate of decrease of the objective function, as shown in Figure
\ref{fig:SpeedupObjDecLsSupAccCspmVsLsAccCspm}.

\begin{figure}[htp]
\centering
\includegraphics[width=0.6\textwidth]{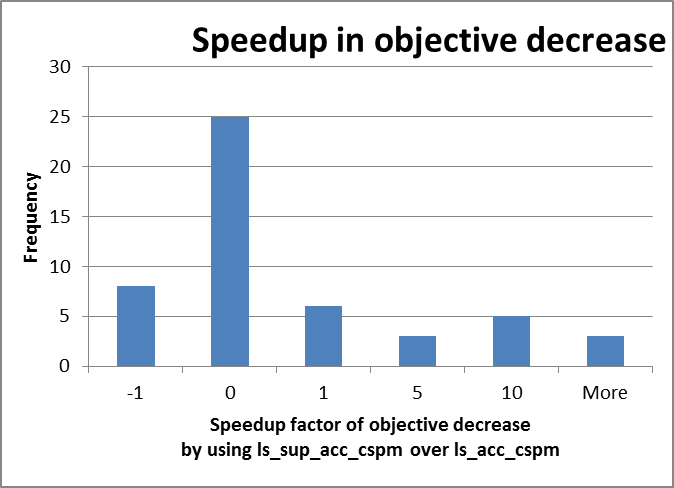}\caption{Factor
of speedup of the objective function decrease by using the superiorized
version of CSPM in the level set scheme for QPS problems. A factor of 0
indicates no difference, negative factors indicate that the superiorized
version exhibited a slower decrease.}%
\label{fig:SpeedupObjDecLsSupAccCspmVsLsAccCspm}%
\end{figure}

For IMRT problems, this decrease was only marginal: the rate of
decrease of the superiorized versions is on average only about 0.5
times faster than the basic versions (note that 0 times would indicate they progress at the same rate).

Hence, the superiorized version of CSPM uses many more projections and
evaluations. However, if these are cheap to compute, then the potential
increase in objective function reduction in early iterations could lead to a
faster approach overall for some problems if the user is willing to stop the
optimization prematurely for practical reasons.

\subsection{Which is the better optimization scheme?}

We compare the level set scheme using the accelerated CSPM and the bisection
scheme with CSPM, as these are the most promising candidates for each optimization scheme given the analysis above.
For the QPS problems, there is no statistically significant
difference in the objective score between the two methods (the p-value of the
two-tailed test is 0.389). However, ls\_acc\_cspm outperforms the bisection
scheme significantly for IMRT problems - even if 4 outliers of the 36 problems
were removed (those which skewed the average quality score of the bisection
scheme to the right). With a p-value of 0.007, ls\_acc\_cspm produces a better
quality score than the equivalent bisection scheme.

Moreover, as Figure \ref{fig:ComplexityBisVsLs} shows for QPS problems, on
average, the bisection scheme requires many more projections and objective
evaluations.
An intuitive explanation to the results is perhaps because in the bisection scheme one may be in an infeasible detecting stage several times, and in each such a stage many calculations are done (which eventually lead one to conclude that infeasibility has been detected). In the level-set scheme an infeasibility stage can happen only one time, and in the other stages usually feasibility is detected.

\begin{figure}[htp]
\centering
\subfigure[Increase in projections]{
\includegraphics[width=0.45\textwidth]{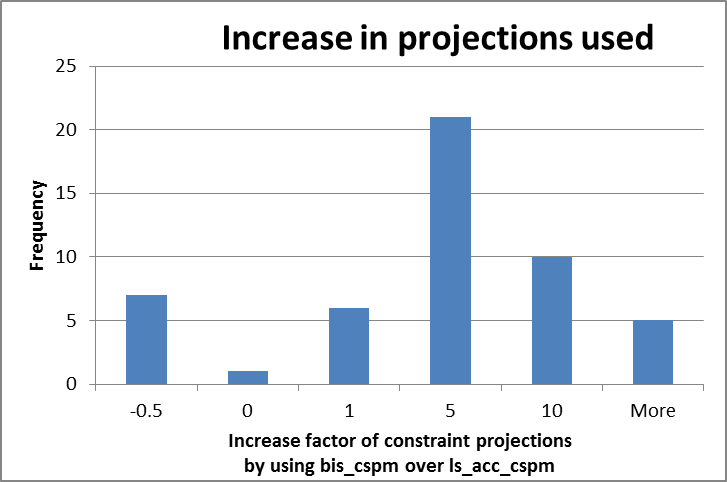}
} \subfigure[Increase in objective evaluations]{
\includegraphics[width=0.45\textwidth]{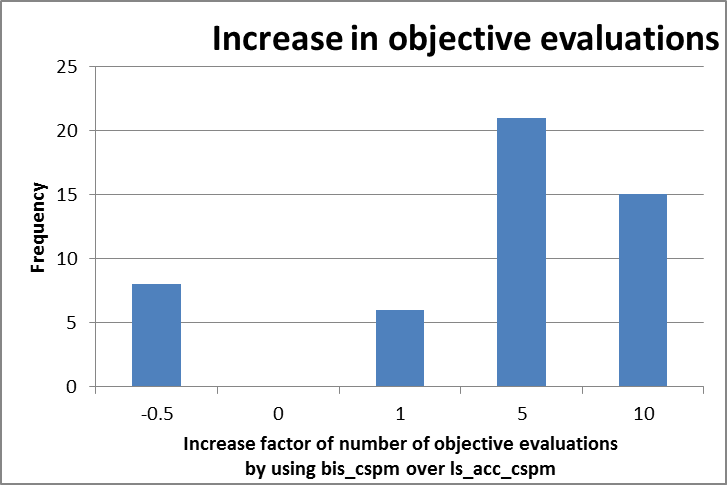}
}\caption{Factor of how much the complexity increases for QPS problems by
using the bisection scheme over the level set scheme. `0' indicates no
difference, and negative values indicate a decrease in complexity.}%
\label{fig:ComplexityBisVsLs}%
\end{figure}

For IMRT problems, the increase is less pronounced, but
consistent. There the increase is up to factor 4.

In terms of rate of objective decrease, the results are mixed. In fact, it
seems that bisection can be expected to decrease the objective a little faster
than the level set scheme. However, Figure
\ref{fig:SpeedupObjDecBisCspmVsLsAccCspm} shows that for the QPS problems this
does not happen very often.

\begin{figure}[htp]
\centering
\includegraphics[width=0.6\textwidth]{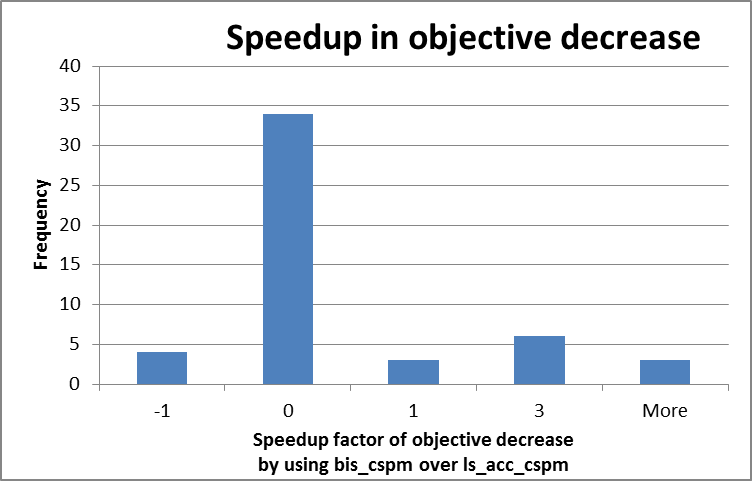}\caption{Factor
of speedup of the objective function decrease by using the bisection scheme
over the level set scheme (QPS problems).}%
\label{fig:SpeedupObjDecBisCspmVsLsAccCspm}%
\end{figure}For IMRT problems, the results are similar, however, it seems as
if a speedup of factor 2 occurred quite frequently.

The level set scheme seems to be the better optimization tool for IMRT
problems in terms of the quality and complexity. But, if one allows superiorization, then this is not always the case and the differences might be very minor, compare for example "ls\_sup\_cspm" and "bis\_sup\_cspmsee" in Tables \ref{table:SummaryQualityScoresAlgorithms_QPS} and \ref{table:SummaryQualityScoresAlgorithms_IMRT}.
Overall, although both strategies are able to obtain similar qualities in the solutions for the QPS
problems, there is a clear advantage of the level set scheme over the
bisection scheme when it comes to complexity.

\section{Concluding remarks and Further research}
\label{sec:conclusion}

Projection methods are known for their computational efficiency and
simplicity. This is the reason we decided to use a well-known reformulation of
a convex optimization problem and apply projection methods within that
general scheme. While at this point the convergence proof for the scheme is
valid only when finite convergent algorithms are used, numerical experiments
show that general convergent algorithms also generate good solutions when the
stopping rule is chosen in an educated way. We believe that the mathematical
validity of this relies on Remark \ref{remark:conv} and is still under investigation.

Another direction we plan to investigate is based on Yamagishi and Yamada
results \cite{yy08}, which show how to replace the subgradient projections by
a more efficient projection when additional knowledge, such as lower bounds is
provided. In addition we plan to study accelerating techniques for projection
methods, for example the recent work of Pang \cite{Pang12} and \cite{Pang13}. Another direction for investigation is the usage of other type of projection methods, for example Bregman projection, see e.g., \cite{CZ97}.\bigskip

\section*{acknowledgement}
We wish to thank the referees for their thorough analysis and review, all their comments and suggestions helped tremendously in improving the quality of this paper and made it suitable for publication. In addition we thank the Associate Editor for his time and effort invested in handling our paper and providing useful remarks. Last but not least, we wish to thank Prof. Yair Censor for his helpful comments and providing useful references.

This work was supported by the Federal Ministry of Education and Research of Germany
(BMBF), Grant No. 01IB13001 (SPARTA). The first author's work was also supported
by ORT Braude College and the Galilee Research Center for Applied Mathematics, ORT Braude College.

\end{document}